\documentclass[a4paper, 10pt]{article}
\usepackage{graphicx}
 
\usepackage{alltt}
 
\usepackage{amsmath}
 
\usepackage[hidelinks, pdftex]{hyperref}
 
\usepackage [T1]{fontenc}

\usepackage{amsmath,amsthm,amssymb}
 \newtheorem{thm}{Theorem}
\newtheorem{lem}{Lemma}
 
\newtheorem{cor}{Corollary}

\newtheorem{prop}{Proposition}

\def\E{\mathbf{E}}

\def\N{\mathbb{N}}
\def\R{\mathbb{R}}
\def\S{\mathbb{S}}
\def\C{\mathbb{C}}
\def\D{\mathbb{D}}
\def\Z{\mathbb{Z}}

\def\k{\varkappa}

\def\e{\varepsilon}

\def\re{{\rm e}}
\def\rd{{\rm d}}

\def\rd{{\rm d}}

\def\E{{\mathbf E}}

\def\e{\varepsilon}

\def\mcM{{\mathcal M}}

\begin{document}

\title{A stochastic process defined via the random\\ permutation divisors} 

\author{Eugenijus Manstavi\v{c}ius}

\maketitle

\footnotetext{{\it AMS} 2000 {\it subject classification.} Primary
60F17;      secondary 60C05. \break {\it Key words and
phrases}. Symmetric group; beta distribution; functional limit theorem; Skorokhod space .}

\begin{abstract}
The normalised partial sums of values of a nonnegative multiplicative function over divisors with appropriately restricted sizes of a random permutation from the symmetric group  define trajectories  of a stochastic process.    We prove a functional limit theorem in the Skorokhod space  when the permutations are drawn uniformly at random. Furthermore, we show that the paths of the limit process almost surely belong to the space of continuous functions on the unit interval and, exploiting the results from  number-theoretical papers, we obtain rather complex formulas for the limits of joint power moments of the process.
   \end{abstract}

\section{Introduction and result} 

Stochastic processes appear in various constructions based upon permutations $\sigma$ taken at random from the symmetric group ${\Sigma}_n$. Sometimes their distribution limit as $n\to\infty$ is the Brownian motion (see \cite{DeLP-85} and \cite{GJBEM-San99}); in other cases, it is some other process with independent or dependent increments (see, e.g. \cite{GJBEM-AnISM02} or \cite{GJBEMVZ-07}). The Poisson--Dirichlet process has received greater attention.  For it, we refer to Sections 5.5, 5.7 and 8.2 of the book  \cite{ABT}. In the present paper, we propose a new type process construction based upon the permutation divisors and prove the functional limit theorem in the  Skorokhod space $\D\big([0,1], \mathcal D\big)$ (see Chapter 3 in \cite{Bil}). 
 The very idea of the problem statement goes back to the number-theoretical paper by the author  and N.\,M. Timofeev \cite{EMNMT-97}. The recent result  by G. Bareikis and the author \cite{GBEM-24} on the mean 
of such processes has been promising.

Recall that $\sigma\in\S_n$ is a one-to-one (bijective) mapping $ \sigma:\N_n \to\N_n:=\{1,\dots, n\}$. It can be represented by the table
\[
     \sigma=\binom{1\; 2 \; \dots\; n\,}{i_1\, i_2\, \dots\, i_n},
     \]
     where $\sigma(r)=i_r$ for $1\leq r\leq n$, or by the digraph $G_\sigma$ with vertex set $V(\sigma)=\N_n$. Its components are oriented cycles. A typical cycle has  a vertex set $V(\k)=\{k_1,\dots, k_j\}\subset\N_n$ defined by 
  \[
        k_1\xrightarrow{\sigma} k_2\xrightarrow{\sigma}\cdots\xrightarrow{\sigma}
         k_j\xrightarrow{\sigma} k_1.
        \]
   Using mapping  multiplication,
 one obtains the unique (up to the order of factors) 
  decomposition of $\sigma$  into cycles $\k_i$ on pairwise disjoint subsets $V(\k_i)$, namely,
  \begin{equation}
           \sigma=\k_1\cdots\k_w.                                                                                                                                                                                                                                                                                                                                                                         \label{1}
  \end{equation}
  Here and in what follows
 $w=w(\sigma)$ denotes the number of cycles and $\N_n=V(\k_1)\cup\cdots \cup V(\k_w)$.  Let us also introduce the empty permutation $\emptyset$ and $\S_0=\{\emptyset\}$. 
A subset  $\delta\subset\{\k_1, \dots,\k_w\}$, including the empty one,  
   is a \textit{divisor} of $\sigma$. Being used to the product expression
  (\ref{1}), we  use the notation $\delta\vert \sigma$ rather than $\delta\subset\sigma$.
Let  $k_j(\delta)$ be the number of cycles in $\delta$ of length $j$ if $1\leq j\leq n$. 
The vector 
\[
\bar k(\delta):=\big(k_1(\delta),\dots, k_n(\delta)\big)
\]
 will be called the \textit{cycle vector} of $\delta$. Observe that $0\leq k_j(\delta)\leq k_j(\sigma)$ for each $j\leq n$ if $\delta\vert \sigma$.
 If $\ell(\bar s):=1s_1+\cdots+ ns_n$ for a vector $\bar s=(s_1,\dots, s_n)\in \Z_+^n$, then
$\ell\big(\bar k(\delta)\big)=\#V(\delta)=:|\delta|$ is the \textit{size} of $\delta$.

A construction of the processes we are interested in is based on the  \textit{multiplicative functions} $q:\S_n\to\R$.  For our present purpose, the definition has to be consistently extended  for the divisors of $\sigma\in \S_n$ as well. Thus, we begin with a family of functions  $q_j:\Z_+\to\C$, $j\in\N$, such that
$q_j(0)=1$ for every $j\in \N$,  and put
\begin{equation}
     q(\delta)=\prod_{j\leq n} q_j(k_j(\delta)), \quad \delta\vert\sigma,\; \sigma\in \S_n.
     \label{multq}
     \end{equation}
 We can say that $q_j(s)$, $0\leq s\leq n$, is the function value prescribed to an $s$-subset  of cycles of length $j$ from  a permutation divisor $\delta$. If $q_j(k)=q_j(1)^{k}$ for each $j\leq n$ and $k\geq 0$, the function $q$ will be called \textit{completely multiplicative}.
Note that the multiplicative functions are \textit{structure dependent}, that is, their values $q(\delta)$ depend only on the vector $\bar k(\delta)$ irrespective of the vertex labels. 
In the sequel, the multiplicative functions $ f,  h$, and  $g$ will have expressions as in (\ref{multq}) with
$f_j(\cdot), h_j(\cdot)$, and $g_j(\cdot)$, respectively.  

  Let $\mcM$ and $\mcM_c$ denote the classes of multiplicative and completely multiplicative functions defined as in (\ref{multq}) for $\delta\vert \sigma$, where $\sigma\in\S_n$.
  Given $g, h\in \mcM$, one can introduce the convolution
  \[
              q(\sigma)=\sum_{\sigma=\delta \tau} g(\delta)h(\tau)= \sum_{\delta\vert\sigma}g(\delta) h(\sigma/\delta).
               \]
  Here  the first summation is over the ordered decompositions of $\sigma$ into the product of divisors. Observe also that the relation
  \begin{equation}
               q_j(k)=\sum_{s=0}^{k} \binom{k}{s} g_j(s) h_j(k-s), \quad j\leq n,
  \label{conv}
  \end{equation}
  holds.   Indeed, for each $j\leq n$, any $\delta\vert \sigma$ has  an $s$-subset, $0\leq s\leq k_j(\sigma)=:k$,
   of the cycles of length $j$ from $\sigma$. The subset can be chosen 
  in $\binom{k}{s}$ ways and, moreover,  the values of $g$ on these subsets coincide, as do the values of $h$ on the $(k-s)$-subsets of the remaining cycles of the same length. Note also that the classes  $\mcM$ and $\mcM_c$ are closed under  convolution. The recalled toolkit will be used in what follows.
  
Given a nonnegative  $g\in\mcM$,  we define the multiplicative function
\begin{equation}
      f(\sigma):=\sum_{\delta\vert \sigma} g(\delta),
      \label{fg}
      \end{equation}
and the family of cumulative distribution functions supported by $[0,1]$:
   \begin{equation}
X_n(\sigma,t):=X_n(\sigma,t;g)=\frac{1}{f(\sigma)}\sum_{\substack{\delta\vert\sigma \\ |\delta|\leq tn}} g(\delta),
\quad \sigma\in \S_n, \; 0\leq t\leq 1.
\label{Xn}
\end{equation}
In particular, 
\[
    X_n(\sigma,t;1)=2^{-w(\sigma)}\#\big\{M\subset \N_n:\; \# M\leq tn,\; \sigma(M)=M\big\},
\quad \sigma\in \S_n, \; 0\leq t\leq 1,
\]
is the relative density of the $\sigma$-invariant subsets $M\subset\N_n$ with cardinality not exceeding $tn$. Indeed,
each  divisor $\delta\vert\sigma$ uniquely determines the vertex set $V(\delta)$ which is $\sigma$-invariant.

Let $\nu_n$ denote the uniform probability measure (Haar) on $\S_n$ and $\E_n$ be the expectation with respect to it. The henceforth used facts about uniformly sampled permutations can be found  in the book \cite{ABT}. We recall just relation (1.3) from page 11; namely,
\[
    \nu_n\big(\bar k(\sigma)=\bar s\big)={\mathbf1}\{\ell(\bar s)=n\} \prod_{j\leq n} \frac1{j^{s_j} s_j!},
    \]
    showing the dependence type of the coordinates of cycle vector. With respect to $\nu_n$, the process $X_n:=X_n(\sigma,t)$ is fairly mysterious. On the other hand, the monotonicity of its trajectories makes the problem  a bit easier to handle. The paper \cite{GBEM-24} witnesses asymptotic regularity of its mean value as $n\to\infty$. 

\begin{thm} \label{BM}
 Let $\vartheta$ be a positive constant and $g\in\mcM_c$ be defined by $g_j(1)=\vartheta$, where $j\leq n$. Then, uniformly in $0\leq t\leq 1$,
\[
  \E_n X_n(t)= \frac1{n!}\sum_{\sigma\in\S_n} X_n(\sigma, t; g)=B\big(t;\theta, 1-\theta\big)+O(n^{-\min\{\theta, 1-\theta\}}), 
   \]
   where $\theta:=\vartheta/(1+\vartheta)$ and 
   \[
   B(t; a,b)=\frac{\Gamma(a+b)}{\Gamma(a)\Gamma(b)} \int_0^t\frac{\rd v}{v^a(1-v)^b},
   \]
   for $0\leq t\leq 1$ and $0<a,b<1$,    is the two-parameter beta distribution function.
   \end{thm}

    In the present paper, we focus on the distribution of $X_n$ defined  by
 \begin{equation}
P_n(A):=\nu_n\cdot X_n^{-1} (A)=\nu_n \big(\sigma\in\S_n:\; X_n(\sigma, \cdot)\in A\big), 
\label{Expect}
\end{equation}
where $A\in \mathcal D$. Formulating the result, we take into account the Addendum \cite{GT-97} to the paper \cite{EMNMT-97}  by G. Tenenbaum.

\begin{thm} \label{Thm1} Let $g\in \mcM_c$ be defined by $g_j(1)=\vartheta$, where  $\vartheta$ is a positive constant and $j\geq 1$. The sequence of distributions $P_n$  converges weakly as $n\to\infty$ to a limiting measure $P$ supported by a subset of $\C[0,1]$.
\end{thm} 
  
  Since $0\leq X_n(t)\leq 1$, by the dominated convergence theorem, the weak convergence 
of measures $P_n\Rightarrow P$ implies convergence of the moments.   

\begin{cor} \label{cor1} Under the conditions of Theorem $\ref{Thm1}$,  for a fixed $l\in\N$ and uniformly in $\bar t:=(t_1,\dots, t_l)\in [0,1]^l$, the following relation for the mixed  moments holds: 
   \begin{equation}
     \lim_{n\to\infty}  \E_n \Big(\prod_{i\leq l} X_n(t_i)\Big) =\int_{\D} \prod_{i\leq l} \varphi(t_i) \rd P(\varphi)=:E(l, \bar t).
  \label{Joint}
  \end{equation}
  \end{cor}

If $l=1$, one recovers the assertion of Theorem \ref{BM} without the remainder term estimate. 
As Proposition \ref{prop1} will show, the finite-dimensional distributions of the limit process are the same as those for the  process appearing firstly in the number-theoretical paper \cite{EMNMT-97}. Hence the limit processes coincide.
From this, we gain the already found  expressions of the limit moments. Let us start with the simple case.

\begin{cor} \label{cor2} Assume that the conditions of Theorem $\ref{Thm1}$ hold, $l=2$, $\bar t=(t,t)$ where $0\leq t\leq 1$, and $\theta=\vartheta/(1+\vartheta)$. Then
   \[
     E(2, \bar t) =
     \begin{cases}I(t,1-\theta) \; &\text{if}\   0\leq t\leq 1/2,\\
     I(1-t,\theta)+2B(t; \theta, 1-\theta)-1\; & \text{if}\  1/2<t\leq 1.\end{cases}
     \]
     Here
     \[
        I(t,a)=\frac{2}{\Gamma(a^2)\Gamma(b^2)\Gamma^2(ab)}\int_0^t\frac{\rd w}{w^{1-ab}}\int_0^{t-w}\frac{\rd v}{v^{1-b^2}}
        \int_0^w\frac{ u^{ab-1}\rd u}{(1-w-v-u)^{1-a^2}}
        \]
        if $0\leq t\leq 1/2$, $0<a<1$, and  $b=1-a$.
      \end{cor}
To check this, it suffices to apply Theorem 2.2 proved by G. Bareikis and A. Ma\v{c}iulis \cite {GBAM-RamJ15}. 
A few computer drawn sketches of $I'_t(t,a)$ are also exposed in this paper.
Generalizing the latter and the previous paper \cite{GBAM-AA12} by the same authors, R. de la Breteche and G. Tenenbaum \cite{BreGT-AAP16} succeeded in writing rather complex formulas  (see (1.7) in their paper) for all $E(l,\bar t)$ in (\ref{Joint}). For the reader's convenience, we include them.

  Let $l\geq 1$ be fixed, $r=2^l-1$, and $\bar v=(v_1,\dots, v_r), \, \bar t=(t_1,\dots, t_r)\in\R^r$.
   In what follows, the inequality $\bar u\leq \bar t$ will mean $u_j\leq t_j$ for each $j\leq r$. Set  $ s(\bar v):=v_1+\cdots+ v_r$.  
 For a nonnegative integer $m$, introduce the base-$2$ digits $d_j(m)\in\{0,1\}$ of $m$, the sum of digits $d(m)=d_0(m)+d_1(m)+\cdots$,  and 
  \[
    u_j(v):=\sum_{1\leq m\leq r} d_j(m) v_m,  \qquad  \bar u(\bar v):=\big(u_1(v), \dots,u_r(v)\big).
    \]
  Define the region
     \[
     \Omega(\bar t):=\big\{\bar v\in[0,1]^r:\;  \bar u(\bar v)\leq \bar t,\, s(\bar v)\leq 1\big\}.
     \]
     
\begin{cor} \label{cor3} Assume that the conditions of Theorem $\ref{Thm1}$ hold, $l\geq 1$, and $\bar t=(t_1,\dots, t_r)\in ]0,1]$. Then
   \begin{align*}
     E(l, \bar t) =\prod_{0\leq m\leq r} &\Gamma\Big(\vartheta^{d(m)}(1+\vartheta)^{-r}\Big)^{-1}\\
     &\quad \times
     \int_{\Omega(\bar t)} \prod_{1\leq m\leq r} v_m^{\vartheta^{d(m)}(1+\vartheta)^{-r}-1}
     \big(1-s(\bar v)\big)^{
     (1+\vartheta)^{-r}-1} \rd \bar v.
     \end{align*}
     \end{cor}
If $l=1$, this equals $B(t;\vartheta/(1+\vartheta), 1/(1+\vartheta)$ as in Theorem \ref{BM}. If $l=2$, we return to  Corollary \ref{cor2}.

 Finally, it is worth reckoning the already mentioned case related to the $\sigma$-invariant subsets.
 
 \begin{cor} \label{cor4} Assume that the process $X_n$ is defined as in $(\ref{fg})$ and $(\ref{Xn})$ via $g(\delta)\equiv1$. Then
  \begin{align*}
  (i)& \quad    \E_n X_n(t) =2\pi^{-1}\operatorname{arcsin}\sqrt{t} +O\big(n^{-1/2}\big), 
  \quad 0\leq t\leq 1;\\
  (ii)&\quad  \text{for all}\; 0<s<t<1,\;  \text{the increments} \; X_n(t)-X_n(s)\;
   \text{converge in }\\
     & \text{ distribution to the discrete random variable whose values are dyadic}
    \\
    & \text{ rational numbers};\\
   (iii)
   &\quad \text{if}\; \varphi(t) \; \text{is a trajectory of the limit process, then}\; 
    \text{P-almost surely}\; \\
   &\quad 
    \varphi'(t)=0\; \text{for every}\; 0<t<1. 
     \end{align*}
   \end{cor}
   
   Claim (i) is just a special case of Theorem \ref{BM}. The properties (ii) and (iii) stem from the notable Tenenbaum's paper \cite{GT-AIF79} and Theorem \ref{Thm1}.
  
  The main needed lemmata are collected in the next section. The proof of Theorem~\ref{Thm1} is presented in Section \ref{sec3}. Namely, Proposition \ref{prop1} shows that the influence of short cycles is negligible, and Proposition \ref{prop2} establishes the convergence of marginal laws of finite order. Proposition \ref{prop3} verifies the tightness criteria for the sequence of measures $\{P_n\}_{n=1}^{\infty}$ 
  even in a stronger form than  needed for the application of Theorem 15.5 in the book~\cite{Bil}.

\section{Lemmata}
From  the general asymptotic theory on the mean values of multiplicative functions, nowadays having a vast literature, we will use the following result.
   
\begin{lem}  \label{Lem1} If $q\in\mcM_c$ satisfies $0\leq q_j(1)\leq 1$ for $j\leq n$, then
\[
    \E_n q\leq \big(\re^{\gamma}+O(n^{-1})\big)\exp\bigg\{\sum_{j\leq n}\frac{q_j(1)-1}{j}\bigg\}.
        \] 
    Here $\gamma$ denotes the Euler--Mascheroni constant.
            \end{lem}
            
    \begin{proof} See $\cite{EM-MoHe17}$. This paper also exposes a possibility to substitute $\re^{\gamma}$ by a smaller quantity.\end{proof}
    
If $1\leq r\leq m$ and $m\geq 2$  are arbitrary integers, then we can uniquely decompose    $\sigma\in\S_m$ into the  product $\sigma=\sigma'\sigma''$, where $\sigma'$ is the so-called $r$-\textit{friable} (smooth)  and  $\sigma''$ is $r$-\textit{free} divisor. Namely, $\sigma'$  comprises  all cycles in $\sigma$ whose lengths do not exceed $r$, while $\sigma''$ contains the remaining ones. Counting such divisors, we will reduce the task to enumerate the $r$-{friable} and the $r$-{free} permutations in $\S_r$ where $r\leq n$. Now we present a few known results on that.

As usual, let $\rho\colon [0,\infty[\to]0,1]$ denote the Dickman--de Bruijn function defined as $\rho(u)=1$ for $[0,1]$ and, for the rest of its range, by the delay differential equation $u\rho'(u)+\rho(u-1)=0$. Recall (see, e.g. Section 5.4 in  \cite{GT-95})  that $\rho(u)=u^{-u+o(u)}$ as $u\to\infty$.  

    \begin{lem} \label{Lem2} If $ 1\leq r\leq m$,   $m\geq 2$, and $u:=m/r$, then
    \[
                   \nu_m\big(\sigma\in\S_m:\;  \sigma\,\, \text{is}\,\, \text{r-friable}\big)=
                   \rho(u)\Big(1+ O\Big(\frac{u\log (u+1)}{r}\Big)\Big).
                   \]
    \end{lem}
      
      \begin{proof} This is Proposition 1.8 in \cite{OGor-23} extending the asymptotic formula valid for $ \sqrt{m\log (m+1)}\leq r\leq m$ established in  the paper \cite{EMRP-EJC16} written jointly with R. Petuchovas.
     \end{proof} 
    
    Let $\omega\colon [0,\infty[\to]0,1]$ denote the Buchstab's function defined as $\omega(u)=1/u$ for $[1,2]$ and by the delay differential equation $(u\omega(u))'=\omega(u-1)$ for $u\geq 2$. Theorem 4 on page 402 in \cite{GT-95} gives that $\omega(u)-\re^{-\gamma}\ll \rho(u)\log^{-1}(u+1)$ if $u\geq 1$.  

    \begin{lem} \label{Lem3} If $ 1\leq r\leq m$  and $u:=m/r\geq 1$, then
    \[
                   \nu_m\big(\sigma\in\S_m:\;  \sigma\,\, \text{is}\,\, \text{r-free}\big)=
                   \exp\Big\{-\sum_{j\leq r}\frac1j\Big\}\Big(\re^\gamma\omega(u)+O(r^{-1})\Big).
                   \]
    \end{lem}
      
      \begin{proof} This is  Theorem 3 from the author's paper \cite{EM-LMR02}. See \cite{RP-Aus18} and \cite{Ford} for a state-of-the-art survey on enumeration of the $r$-free permutations. 
         \end{proof}

\section{Proof of Theorem \ref{Thm1}}\label{sec3}

We split the \textit{proof}  into three parts.

\subsection{Long cycles are essential} 
 
 Let us  discover the role of the divisors having long cycles.
If  $0<\e<1$, we let
\[ 
\sigma'(\e)=\prod_{\k\vert\sigma\atop |\k|\leq \e n} \k, \qquad \sigma(\e)=\prod_{\k\vert\sigma\atop\e n<|\k|\leq n}\k 
\]  
denote the $(\e n)$-friable  and  $(\e n)$-free divisors, respectively. Introduce the process
\[
     X_n(\sigma(\e), t)=\frac1{f(\sigma(\e))}\sum_{\delta\vert\sigma(\e)\atop |\delta|\leq tn} g(\delta), \quad 0\leq t\leq 1.
     \]
For brevity, let $\alpha=(\log(1/\e))^{1/2}$.
  
\begin{prop} \label{prop1} There exist absolute positive constants $c_0$, $\e_0$ and $C$ such that
\[
\nu_n\big(X_n(\sigma, t)\not=X_n(\sigma(\e), t)\big)\leq C \alpha^{-1}
\]
  uniformly in $\e^c\leq t\leq 1-\e^c$, provided that $0<c<c_0$, $0<\e<\e_0$ and $n\geq n_0(\e)$. 
\end{prop}

\begin{proof} We adopt the original arguments used in the number-theoretical paper \cite{EMNMT-97}. We may start with $\e_0\leq \re^{-1}$ and $ n_0(\e)> \e^{-1}\geq\re$ and refine the choice in the proof process. 
 
Each  $\delta\vert \sigma'(\e)\sigma(\e)$ splits into a product of two divisors $\delta=\delta_1\delta_2$ such that
$\delta_1\vert \sigma'(\e)$ and $\delta_2\vert \sigma(\e)$. Hence
  \begin{align*}
X_n(\sigma, t)&= \frac1{f(\sigma'(\e))}\sum_{\delta_1\vert\sigma'(\e)} g(\delta_1) \cdot \frac1{f(\sigma(\e))}
\sum_{\delta_2\vert \sigma(\e)\atop |\delta_2|\leq nt}g(\delta_2)\\
&\quad -
\frac1{f(\sigma)}\sum_{\delta_1\vert\sigma'(\e)} g(\delta_1)\sum_{\delta_2\vert \sigma(\e)\atop  tn -|\delta_1|<|\delta_2|\leq nt} g(\delta_2)\\
&=:
X_n(\sigma(\e), t)- Y_n(\sigma, t).
\end{align*}
  
  The largest $(\e n)$-friable divisor of $\sigma$ is $\sigma'(\e)$.
Observe that the subset of $\sigma$ having comparatively large $\sigma'(\e)$ is sparse. Indeed,
\begin{align*}
     \nu_n\big(\sigma:\; |\sigma'(\e)|>\e\alpha n\big)&=
      \nu_n\bigg(\sum_{j\leq \e n} j k_j(\sigma)>\e \alpha n\bigg)\leq 
      \frac{1}{\e \alpha n}\sum_{j\leq \e n} j \E_n k_j(\sigma)\\
      &=
            \frac{1}{\e \alpha n} \lfloor{\e} n\rfloor\leq \frac1{\alpha}.
\end{align*}
Hence
\begin{equation}\label{nu}\begin{split}
\nu_n\big(Y_n(\sigma,t)\not=0\big)
&\leq \alpha^{-1} +
\nu_n\bigg(\sigma:\; |\sigma'(\e)|\leq \e\alpha n, \, \exists \delta_2\vert\sigma(\e),\, 
tn -|\sigma'(\e)|< |\delta_2|\leq tn\bigg)\\
&=:\alpha^{-1} +\mu_n(t).
\end{split}\end{equation}  

Let us focus on the $\sigma$'s counted in $\mu_n(t)$. If $\sigma(\e)=\delta_2 \delta_2'$, then solving the inequalities between the parentheses in $\mu_n(t)$,  by virtue of 
\[
   |\delta_2|+|\delta'_2|=n-|\sigma'(\e)|,
   \]
 we have 
 \begin{equation}
(t-\e\alpha)n \leq |\delta_2|\leq tn, \qquad (1-t-\e\alpha)n \leq |\delta'_2|\leq (1-t)n
\label{neld}
\end{equation}
provided that $\e\alpha\leq t\leq 1-\e\alpha$.
Now, it is essential to verify that at least one of such $\delta_2$ or $\delta'_2$ has a comparatively small number of cycles. By the definition of the number-of-cycles function, 
\[
    w(\sigma(\e))=\sum_{\e n<j\leq n}  k_j(\sigma).
\]
The effective inequality for the second moment of an arbitrary  completely additive function has been established by J. Klimavi\v{c}ius and the author in \cite{JKEM-AnBud18}. For the particular case, we have
\begin{equation}
  \frac1{n!}\sum_{\sigma\in\S_n} \big(w(\sigma(\e))- h(\e n, n) \big)^2= h(\e n, n)-
    \sum_{\e n< i, j\leq n\atop i+j>n} \frac1{ij}\leq  h(\e n, n), \quad n\geq 2,
  \label{KM}
  \end{equation}
  where
  $
  h(y,x):=\sum_{y<j\leq x}1/j$ satisfying the inequality $\big|  h(y,x)-\log(x/y)\big|\leq 1/y$.
  Applying Chebyshev's inequality and (\ref{KM}), we obtain 
\begin{equation}
\nu_n \big(w(\sigma(\e))>(3/2) \alpha^2\big)\leq 
\nu_n \big(w(\sigma(\e))-h(\e n, n)> \alpha^2/4\big)\leq 32 \alpha^{-2}
   \label{KM1}
  \end{equation}
if $n \geq \e^{-1}\alpha^{-2}$. Thus, for all but $O( n! \alpha^{-2})$ permutations $\sigma\in\S_n$,  one of the above $\delta_2$ or $\delta'_2$ has  no more than $(3/4)\alpha^2$ cycles. 
In either  case, for the most of $\sigma$ counted in the frequency $\mu_n(t)$, we obtain a decomposition $\sigma=\delta \tau(\e)$, where $\tau(\e)$ is  $(\e n)$-free and, with $\delta(\e)=\delta_2$ or $\delta(\e)=\delta'_2$, the divisor $\delta:=\delta(\e)\sigma'(\e)$ belongs to the set
\[
     \Delta_t:=\big\{\delta\vert\sigma:\;   w(\delta(\e))\leq (3/4)\alpha^2, \, (t-\e\alpha)n\leq |\delta|\leq (t+\e\alpha)n \big\}
    \]
for  $ t$ (or $(1-t)$) from $[\e\alpha,\, 1-\e\alpha]$, as indicated in (\ref{neld}). In the shorter interval for  $t$, we have to establish the uniform estimate. Therefore, we proceed with
\begin{equation}
  \mu_n:= \max_{\e^c\leq t\leq 1-\e^c}\mu_n(t)\ll \alpha^{-2}+\max_{\e^c\leq t\leq 1-\e^c}\frac1{n!}\sum_{\sigma=\delta\tau(\e)\in\S_n\atop\delta\in \Delta_t}  1.
   \label{free}
   \end{equation}
 where $0<c<1/2$.
   
    As described in \cite{GBEM-24} or \cite{FlSed}, Chapter II, the sum over decompositions in (\ref{free})  can be reduced to summation over permutations belonging to respective symmetric groups  of lower order. For that, the vertex labels in $\tau(\e)$   and simultaneously the labels in $\delta$ belonging to  $\N_{n}$ can be substituted by the numbers from $\N_{|\tau(\e)|}$  and $\N_{|\delta|}$ so that the former orders of the labels in either of the divisors is  preserved. If $|\delta|=k$, then exactly  $\binom{n}{k}$ of the  pairs $(\delta, \tau(\e))$ are reduced to one pair $(\delta, \tau(\e))$, with $\delta\in\S_k$ and $\tau(\e)\in\S_{n-k}$. Here we are leaving the same notation after the relabelling since the reduction does not change the cycle structure of divisors; in particular, neither the sizes and nor $w(\delta(\e))$. Consequently, (\ref{free}) attains the form
 \begin{equation}
   \mu_n \ll \alpha^{-2}+\max_{\e^c\leq t\leq 1-\e^c}\frac1{n!}\sum_{(t-\e\alpha)n\leq k\leq (t+\e\alpha)n}\binom{n}{k}
   \sum_{\delta\in\S_k\atop w(\delta(\e))\leq (3/4)\alpha^2} \sum_{\tau(\e)\in\S_{n-k}}  1.
   \label{free1}
  \end{equation}  
   The  innermost sum counts the $(\e n)$-free permutations in the symmetric group $\S_{n-k}$.
    Since 
    \[
    (n-k)/(\e n)\geq  (1-t-\e\alpha)/\e\geq \e^{c-1}/2\geq 1,
    \]
    for $0<c\leq 1/2$ assuring $\e\alpha\leq \e^{c}/2$ if $\e\leq \re^{-2}$,
    we can apply Lemma \ref{Lem3} to get
    \[
     \frac1{(n-k)!} \sum_{\tau(\e)\in\S_{n-k}}  1 \ll \frac1{\e n}.
     \]
   Inserting this into estimate (\ref{free1}), we obtain
\begin{equation}
   \mu_n \ll \alpha^{-2}
 +
   \max_{\e^c\leq t\leq 1-\e^c}\frac1{\e n}\sum_{(t-\e\alpha)n\leq k\leq (t+\e\alpha)n} \nu_k \bigg(\delta\in\S_k: \; w(\delta(\e))\leq (3/4)\alpha^2\bigg).
   \label{free2}
   \end{equation}
The summation is over large $k$, namely,   $k\geq  (t-\e\alpha)n\geq \e^c  n/2$. Therefore, we can again use inequality (\ref{KM}) with $k$ instead of $n$, centralizing $w(\delta(\e))$ by $h(\e n, k)$.
The frequency under the sum does not exceed
\[
\nu_k:=\nu_k \big(\delta\in\S_k: \; w(\delta(\e))-h(\e n, k)\leq -(1/4)\alpha^2 +R\big),
\]
where
\[
R=\alpha^2 -\log\frac{k}{\e n} +\frac1{\e n}=\log\frac{n}{k}+\frac1{\e n}\leq  c\alpha^2+\log 2+\frac1{\e n}.
\]
If $c<1/24=:c_0$, $\e<2^{-24}=:\e_0$, and $n\geq n_0(\e):= 24\e^{-1} \alpha^{-2}$, then $R\leq
 \alpha^2/8$. Hence
\[
 \nu_k\leq \nu_k \big(\delta\in\S_k: \; w(\delta(\e))-h(\e n, k)\leq -(1/8)\alpha^2 \big)\ll
     \alpha^{-4} h(\e n,k)\ll     \alpha^{-2}.
     \] 
    Combining the last estimate with (\ref{free2}), we obtain
\[
\mu_n\ll \alpha^{-2}+(\e n)^{-1}\cdot  \alpha^{-2}\cdot \e \alpha n\ll\alpha^{-1}.
\]
    
    Recalling inequalities (\ref{free}) and (\ref{nu}), we complete the proof of Proposition \ref{prop1}.
 \end{proof}

\subsection{Convergence of finite-dimensional distributions} 

This step is devoted to the finite-dimensional distributions of the process $X_n(t)$.
   Let $l\in\N$ and   $0\leq t_1<\cdots<t_l\leq 1$ be arbitrary fixed numbers. Define the vectors  $T:=(t_1,\dots,t_l)$,  $U:=(u_1,\dots,u_l)\in[0,1]^l$, and the distribution function
   \[
        F_n(U, T):=\nu_n\big( X_n(t_1)\leq u_1, \, \dots, X_n(t_l)\leq u_l\big).
        \]
   
      \begin{prop} \label{prop2} For all vectors $T$ and $U$, the distribution function
   $    F_n(U, T)$ converges as $n\to\infty$ to an $l$-dimensional distribution function.
   \end{prop}
   
   \begin{proof} Since $0<X_n(\sigma,t)\leq X_n(\sigma,1)=1$, $\sigma\in\S_n$ and $0\leq t\leq 1$,  without loss of generality, we can assume that $t_l<1$ and $x_1>0$. 
 To settle the case $t_1=0$, we check that, according to (\ref{conv}),  the function $1/f\in \mcM_c$ is defined by $f_j(1)=1+g_j(1)=1+\vartheta$  for every $j\in\N$. Using this, we   evaluate the difference
 \begin{align*}
 0\leq &\nu_n\big(X_n(t_2)\leq u_2,  \dots, X_n(t_l)\leq u_l\big)\\
 &\quad -
 \nu_n\big( X_n(0)\leq u_1, X_n(t_2)\leq u_2,  \dots, X_n(t_l)\leq u_l\big)\\
 &\leq
     \nu_n\big(X_n(0)>u_1\big)\leq \frac1{u_1} \E_n (1/f)\ll n^{-c(\vartheta)}.
     \end{align*}
  In the last step, we applied  Lemma \ref{Lem1}. Here $c(\vartheta)$ is a positive constant depending on $\vartheta$ only. 
Thus, for $u_1=0$, our task would reduce to the $(l-1)$-dimensional problem.
   
   Henceforth let $0<\e<\min\big\{ t_1^{1/c_0}, \,(1- t_l)^{1/c_0}\big\}$, where $c_0$ has been found in Proposition \ref{prop1}. Introduce the distribution functions
   \begin{align*}
        G_n(U, T)&:=\nu_n\big(\sigma:\; X_n(\sigma(\e),t_1)\leq u_1, \, \dots, X_n(\sigma(\e),t_l)\leq u_l\big)
        \\
        &=:\nu_n\big(\sigma:\; \overline{X}_n(\sigma(\e), T)\in A\big),
        \end{align*}
        where the process  $X_n(\sigma(\e), t)$ has been defined in Part 1, 
        \[ 
            \overline{X}_n(\sigma(\e), T):=
        \big(X_n(\sigma(\e),t_1), \dots, X_n(\sigma(\e),t_l)\big), \quad A:=\prod_{i\leq l}]0, u_i[ \subset 
        ]0, 1[^l.
        \]
     By Proposition \ref{prop1},  
          \begin{align*}
        F_n(U, T)&=   G_n(U, T)+O(\alpha^{-1})=\\
        &=
        \frac1{n!}\sum_{\sigma=\sigma(\e)\delta\in\S_n\atop \delta\, \text{is}\,  (\e n)-\text{friable} }
        {\mathbf 1}\big\{\overline{X}_n(\sigma(\e), T)\in A\big\}+O(\alpha^{-1}).
        \end{align*}
        
As in derivation of (\ref{free1}), we can apply reduction of labels (then $\binom{n}{k}$ of the $(\e n)$-free permutations $\sigma(\e)$ reduce to one $\tau\in\S_k$) and rewrite
          \begin{align*}
        F_n(U, T)&= 
        \sum_{\e n\leq k\leq n} \frac1{k!}\sum_{\tau\in\S_k} 
        {\mathbf 1}\big\{\overline{X}_n(\tau, T)\in A\big\}\\
        &\quad \times
        \nu_{n-k}\big(\delta\in\S_{n-k}:\; \delta\,\text{is} \, (\e n)-\text{friable}\big)+
        O(\alpha^{-1})\\
        &=
        \sum_{\e n\leq k\leq n} \frac1{k!}\sum_{\tau\in\S_k} 
        {\mathbf 1}\big\{\overline{X}_n(\tau, T)\in A\big\} \rho\Big(\frac{n-k}{\e n}\Big) + O(\alpha^{-1}),
        \end{align*}
        by Lemma \ref{Lem2} with $m=n-k\geq n(1-\e)$ and $r=\lfloor\e n\rfloor$, taking into account that $\rho(u)\ll u^{-u/2}$ for $u=(n-k)/(\e n)\geq \e^{-1}/2$. Similarly, here we can get rid of $\tau$ having  cycles with repeated lengths. For that, we can use the  estimate
        \begin{align*}
        &\sum_{\e n\leq k\leq n} \nu_k\big(\tau\in\S_k,\; \exists j\, \text{such\, that}\, k_j(\tau)\geq 2\big)\\
        &\leq
         \frac12 \sum_{\e n\leq k\leq n} \sum_{\e n\leq j\leq k}\E_k\Big(k_j(\tau)\big( k_j(\tau)-1\big)\Big)\\
         &=
         \frac12 \sum_{\e n\leq k\leq n} \sum_{\e n\leq j\leq k}\frac1{j^2}\ll \frac1{\e}.
         \end{align*}   
         Multiplied by $\rho(\e^{-1}/2)$, this quantity gives also the remainder $O(\alpha^{-1})$.
        So after  simplifications,  we arrive at
        \begin{align}
             &F_n(U, T)\nonumber\\
             &=
                \sum_{\e n\leq k\leq n}\rho\Big(\frac{n-k}{\e n}\Big) \nu_k\Big(\tau\in\S_k:\; k_j(\tau)\leq 1, \, \e n\leq j\leq k,\; \overline{X}_n(\tau, T)\in A\Big) 
                  + O(\alpha^{-1})\nonumber\\
        &=  
         \sum_{m\leq \e^{-1}}  \sum_{\e n\leq k\leq n} \rho\Big(\frac{n-k}{\e n}\Big) \nu_k\Big(\tau\in\S_k(m):\; \overline{X}_n(\tau, T)\in A\Big) 
                 + O(\alpha^{-1}).
      \label{FnUT}
      \end{align}
       Here $\S_k(m)$ denotes the subset of $(\e n)$-free permutations $\tau$ in $\S_k$ which have exactly~$m$, $1\leq m\leq \e^{-1}$, cycles of different lengths. We further intend to change the summation over $k$ by a summation over $m$-tuples of the cycle lengths appearing in all accounted~$\tau$.
      
      Let $\tau\in \S_k(m)$ and its cycle lengths  in the fixed order be  $(j_1,\dots, j_m)\in \rfloor\e n,n]^m$ provided that $j_1+\cdots+j_m=k$. 
   The number of such $\tau$ equals   $k!/ (j_1\cdots j_m)$.
          What is the meaning of the condition $\overline{X}_n(\tau, T)\in A$ in terms of the cycle lengths $ (j_1,\dots, j_m)$ for the given $\tau$?
          
     The following observations hold for any such $\tau$.  Firstly,  $f(\tau)=(1+\vartheta)^m$. Secondly, the divisors $\delta\vert \tau$ are enumerated by the vector of indicators $(i_1,\dots, i_m)$, where $i_r=1$ if the cycle of length $j_r$ appears in $\delta$. Hence
     $|\delta|=i_1j_1+\cdots+i_mj_m$. Thirdly, the vector $(j_1/n,\dots,j_m/n)$ necessarily belongs to the intersection, denoted by $D_m(\e)$, of the following two sets of the vectors $\bar x=(x_1,\dots, x_m)$:
          \[
     \big\{\bar x\in \rfloor\e,1]^m:\;  x_1+\cdots +x_m\leq 1\big\}  
      \]
    and
    \[
    \bigcap_{j\leq l}\bigg\{\bar x:\; \sum_{(i_1,\dots, i_m)\in\{0,1\}^m} \vartheta^{i_1+\dots + i_m} 
    {\mathbf 1}\big\{i_1x_1+\cdots + i_m x_m\leq t_j\big\}\leq u_j(1+\vartheta)^m\bigg\}.
    \]  
    
    Having all this in mind, we obtain from (\ref{FnUT})
    \begin{align}
    F_n(U, T)&=\sum_{m\leq \e^{-1}}  \sum_{ (j_1/n,\dots, j_m/n)\in D_m(\e)} \rho\bigg(\frac{1}{\e }
    \Big(1-\sum_{r\leq m}\frac{j_r}{n}\Big)\bigg) \frac1{j_1\cdots j_m} 
                 + O(\alpha^{-1})\nonumber\\
     &=
     \sum_{m\leq \e^{-1}} \int_{D_m(\e)} \rho\Big(\frac{1}{\e }
    \Big(1-\sum_{r\leq m}x_r\Big)\Big) \frac{\rd x_1\cdots \rd x_m}{x_1\cdots x_m} 
                 + O(\alpha^{-1}) +o_\e(1)           
                 \label{FnUT1}
      \end{align}
  as $n\to\infty$,      by the definition of the $m$-dimensional Riemann integral. We have arrived at the situation described on page 6 of the paper \cite{EMNMT-97}. Letting successively $n\to\infty$ and  $\e\to 0$, by virtue of the notation $\alpha=(\log(1/\e))^{-1/2}$, we verify that 
  \[
      \limsup_{\e\to0} I(\e)\leq  \liminf_{n\to\infty} F_n(U, T)\leq \limsup_{n\to\infty}F_n(U, T)\leq       \liminf_{\e\to0} I(\e),
       \]
       where $I(\e)$ denotes the sum of integrals in the above relation. This shows that, as claimed in Proposition \ref{prop2}, the limits 
              $\lim_{\e\to0}\, I(\e)=\lim_{n\to\infty}\, F_n(U, T)$
exist for all vectors $T$ and $U$. 
        \end{proof}
  
\subsection{Tightness}

This part is devoted to showing the tightness of the sequence of distributions $\{P_n\}_{n=1}^\infty$. 
    We take advantage of the  G. Tenenbaum's idea \cite{GT-97}  to verify a stronger tightness criterion than that used in the Skorokhod space $\D[0,1]$. According to Theorem 15.5 of \cite{Bil}, the following assertion  also assures that the  weak limit $P$ of a subsequence $\{P_{n'}\}$ as $n'\to\infty$ is supported by a subset of the space $\C[0,1]$.
    
    \begin{prop} \label{prop3} For every $0<a,\epsilon<1$,
     \[
                 \nu_n\Big( \sup_{|s-t|\leq a}\big|X_n(t)-X_n(s)\big|\geq \epsilon\Big)\ll\epsilon^{-1} a^{\vartheta/(1+\vartheta)}.  
  \]
  \end{prop}
       
    \begin{proof}    By virtue of the monotonicity of the process $X_n(t)$, $0\leq t\leq 1$, the modulus of continuity is
   \[
       \sup_{|s-t|\leq a}\big|X_n(t)-X_n(s)\big|= \sup_{ 0\leq t\leq 1-a}\big|X_n(t+a)-X_n(t)\big|=:Q_n(\sigma, a),
       \]
       where $0\leq s\leq t\leq 1$ and $0<a<1$.
   Thus, the estimate of Proposition \ref{prop3} will follow from
   \begin{equation}
\E_n Q_n(\sigma, a)\ll a^{\vartheta/(1+\vartheta)}.
        \label{Qn}
        \end{equation}
   
    For a given $\sigma\in\S_n$, the introduced $Q_n(\sigma, a)$ is just the  concentration function of the random variable taking  value $|\delta|/n$ with  probability $g(\delta)/f(\sigma)$  when $\delta\vert \sigma$.
   By Lemma~2.6.1 in Chapter III of the book \cite{GT-95},
   \[
         Q_n(\sigma,a)\leq 3an\int_0^{1/(an)} \big|G(\sigma, v)\big| \rd v,
         \]
   where
   \[
          G(\sigma, v)=\frac1{f(\sigma)}\sum_{\delta\vert \sigma} g(\delta) \re^{iv|\delta|}, \quad v\in\R.
          \]
   Consequently, 
   \begin{equation}
     \E_n Q_n(\sigma, a)\ll an \int_0^{1/(an)} \E_n\big|G(\sigma, v)\big| \rd v.
     \label{EQn}
     \end{equation}
  Since $|G(\sigma, v)|$  is a bounded completely multiplicative function on $\S_n$ defined by $G_j(1,v)=|1+\vartheta \re^{ijv}|/ (1+\vartheta)$, where $j\in\N$ and $v\in\R$ is a parameter,    
  Lemma \ref{Lem1} yields 
   \begin{align*}
   \E_n\big|G(\sigma, v)|&\ll \exp\bigg\{\sum_{j\leq n}\Big(\frac{|1+\vartheta\re^{i v j}|}{1+\vartheta}-1\Big)
   \frac1j\bigg\}  \\
   &\leq
   \exp\bigg\{-\frac{\vartheta}{1+\vartheta}\sum_{j\leq n}\frac{1-\cos v j}{j}\bigg\}. 
   \end{align*}
  Here we have applied the inequality
   \[
        |1+r\re^{ix}|/(1+r)\leq 1-r(1-\cos x)/(1+r)^2, \quad 0\leq r\leq 1,\, x\in\R,
        \]
        presented on  page 17 of the paper \cite{GT-97} with $r=\vartheta$ and $x= vj$. Using the equality 
                \[
        \sum_{j\leq n}\frac{1-\cos v j}{j}=\log\frac{|1-\re^{-1/n+iv}|}{ 1-\re^{-1/n}} +O(1),
        \]
 we obtain
 \[  
   \E_n\big|G(\sigma, v)|\ll \Big(\frac{ 1-\re^{-1/n}}{|1-\re^{-1/n+iv}|}\Big)^{\vartheta/(1+\vartheta)}
   \leq \Big(\frac{ 1}{n|1-\re^{-1/n+iv}|}\Big)^{\vartheta/(1+\vartheta)}
   \]
with an absolute constant in the symbol $\ll$. Returning to (\ref{EQn}), we proceed
as follows:\begin{align*}
     \E_n Q_n(\sigma, a)&
     \ll a n^{1/(1+\vartheta)}\int_0^{1/(an)} \frac{\rd v}{|1-\re^{-1/n+iv}|^{\vartheta/(1+\vartheta)}}\\
     &\ll 
   a+ a n^{1/(1+\vartheta)}\int_{1/n}^{1/(an)} v^{-\vartheta/(1+\vartheta)} \rd v\ll a^{\vartheta/(1+\vartheta)}.
  \end{align*}
  This is the estimate (\ref{Qn}). 
  \end{proof}
  
   {\bf Concluding remark.} Doing more cumbersome work, one can extend Theorem \ref{Thm1} in two ways. Firstly, the function $g$ can be taken from $\mcM$, not necessarily from $\mcM_c$. It can be proved that the values $g_j(k)$ for $k\geq 2$ and $j\leq n$ are negligible. Adopting Tenenbaum's  \cite{GT-97} argument, one can replace the condition on $g_j(1)=\vartheta>0$ so that their weighted sums over $\e n\leq j\leq n$  behave in the needed manner. An analytic technique in the style of Bareikis and Ma\v{c}iulis \cite{GBAM-RamJ15} works here as well.
   Secondly, as seen from  Theorem 3 of the paper \cite{GBEM-24}, analogous results could be established for permutations drawn according to the Ewens probability measure on $\S_n$.






\vskip 0.5 true in
Affiliation: Institute of Mathematics, Vilnius University, Naugarduko str. 24, LT-03225 Vilnius, Lithuania;
 email: eugenijus.manstavicius@mif.vu.lt


\begin{thebibliography}{99}
\bibliographystyle{APT}
\footnotesize

 \bibitem{ABT} {\sc Arratia, R., Barbour, A. and Tavar\'e, S.} (2003), {\em Logarithmic
        Combinatorial Structures: A Probabilistic Approach}. EMS
        Monographs in Mathematics, EMS Publishing House, Z\"urich.

\bibitem{GJBEM-San99} {\sc Babu, G.J. and Manstavi\v{c}ius, E.} (1999). Brownian motion for random permutations, {\em Sankhya, Ser. A}. {\bf 61,} 312--327.


\bibitem{GJBEM-AnISM02} {\sc Babu, G.J. and Manstavi\v{c}ius, E.} (2002). Limit processes with independent increments for the Ewens sampling formula. {\em Ann. Inst. Statist. Math.}
{\bf 54,} 607--620.

\bibitem{GJBEMVZ-07}  {\sc Babu, G.J., Manstavi\v{c}ius, E. and Zacharovas, V.} (2007). Limiting processes with dependent increments for measures on the symmetric group of permutations, {\em Advanced Studies in Pure Mathematics. {\bf49}. Proc. Int. Conf. “Probability and Number Theory”, Kanazawa, 2005}. Math. Soc. Japan, Tokyo,  41–--67.

 \bibitem{GBAM-AA12} {\sc Bareikis, G. and Ma\v{c}iulis, A.} (1992). 
 Ces\'{a}ro means related to the square of the divisor function. {\em Acta Arith.} {\bf156.1,} 83--99. 
    
\bibitem{GBAM-RamJ15} {\sc Bareikis, G. and Ma\v{c}iulis, A.} (2015).    On the second moment of an arithmetical process related to the natural divisors. {\em Ramanujan J.}. {\bf 37,} 1--24. 

 \bibitem{GBEM-24} {\sc Bareikis, G. and Manstavi\v{c}ius, E.} (2024). Construction of the beta distributions using the random  permutation divisors, {\em Nonlinear Analysis: Modelling and Control}. {\bf29,} 189--204.

\bibitem{Bil} {\sc Billingsley, P.} (1968).  {\em Convergence of Probability Measures}. Wiley \& Sons, New York.

\bibitem{BreGT-AAP16} {\sc De la Bret\'eche, R. and Tenenbaum, G.} (2016). Sur les processus arithm\'etiques li\'es aux diviseurs.  {\em Adv. Appl. Probab.} {\bf4}(A), 63--76.
  
  \bibitem{DeLP-85} {\sc DeLaurentis, J. M. and Pittel, B. G.} (1985). Random permutations and the Brownian motion, {\em Pacific J. Math.}. {\bf 119,} 287--301.

  \bibitem{FlSed} {\sc Flajolet, Ph. and Sedgewick, R.} (2008). { \em Analytic
Combinatorics}. Cambridge University Press, Cambridge.

\bibitem{Ford} {\sc Ford, K.} (2022).  Cycle type of random permutations: A toolkit, {\em Discrete Analysis}. {\bf 9,} 36 pp.

\bibitem{OGor-23} {\sc  Gorodetsky, O.} (2023). Uniform estimates for smooth polynomials over finite fields, {\em Discrete Analysis}. {\bf 16}, 32 pp.

\bibitem{JKEM-AnBud18} {\sc Klimavi\v{c}ius, J.  and Manstavi\v cius, E.} (2018). The Tur\'an--Kubilius inequality on permutations, {\em Annales Univ. Sci. Budapest., Sect. Comp.}, {\bf48,} 45--51.



\bibitem{EM-LMR02} {\sc Manstavi\v{c}ius, E.} (2002). On permutations missing short cycles. {\em Lietuvos matem. rink.} (Special issue). {\bf 42,} 1--6.
 
\bibitem{EM-MoHe17} {\sc Manstavi\v{c}ius, E.} (2017).  On mean values of multiplicative functions on the
symmetric group. {\em Monatsh. Math.} {\bf 182,} 359--376.

\bibitem{EMRP-EJC16} {\sc Manstavi\v{c}ius, E. and Petuchovas, R.} (2016). Local probabilities for random permutations without long cycles. {\em Electronic J. Comb.}  {\bf23,} \#P1.58, 25 pp.

  \bibitem{EMNMT-97} {\sc Manstavi\v{c}ius, E. and Timofeev, N.M.} (1997). Functional limit theorem related to natural divisors.  {\em Acta Math. Sci. Hungaricae}.  {\bf 75,} 1--13.
      
  \bibitem{RP-Aus18} {\sc Petuchovas, R.} (2018). Asymptotic estimates for the number of permutations without short cycles. {\em Australasian J. of Comb.} {\bf 72,} 1--18.
  
  \bibitem{GT-AIF79} {\sc Tenenbaum, G.} (1979). Lois de r\'epartition des diviseurs, 4.  {\em Ann. Inst. Fourier}. {\bf 29,} 1--15.
  
   \bibitem{GT-95} {\sc Tenenbaum, G.} (1995). {\em Introduction to Analytic and Probabilistic Theory of Numbers}. Cambridge Studies in Mathematics: 46. Cambridge University Press, 1995.
     
  \bibitem{GT-97} {\sc Tenenbaum, G.} (1997). Addendum to the paper of E.~Manstavi\v{c}ius and N.M.~Timofeev "Functional limit theorem related to natural divisors". {\em Acta Math. Sci. Hungaricae}.  {\bf 75,} 15--22.
  
\end{thebibliography}
\end{document}